\documentclass[11pt]{amsart}
\usepackage{amssymb}
\usepackage{graphicx}
\usepackage{url}
\usepackage{cite}
\usepackage[left=3cm,right=3cm,top=3cm,bottom=3cm]{geometry}
\usepackage{color}
\usepackage{tikz-cd}
\usepackage{mathtools}
\usepackage[new]{old-arrows}
\usepackage{amsmath}
\usepackage{amsthm}
\usepackage{tikz}
\usepackage{amscd}
\usepackage{latexsym}
\usepackage{epsfig}
\usepackage{graphicx}
\usepackage{psfrag}
\usepackage{caption}
\usepackage{rotating}
\usepackage{mathtools}
\usepackage{enumitem}
\usepackage{bigints} 
\usepackage{longtable}
\usepackage[normalem]{ulem}

\allowdisplaybreaks[3]

\makeatletter
\@namedef{subjclassname@1991}{2020 Mathematics Subject Classification}
\makeatother

\newtheorem{theorem}{Theorem}
\newtheorem*{theorem*}{Theorem}

\newtheorem{corollary}[theorem]{Corollary}
\newtheorem{lemma}[theorem]{Lemma}

\theoremstyle{definition}
\newtheorem{definition}[theorem]{Definition}

\newtheorem{remark}[theorem]{Remark}
\newtheorem*{question*}{Question}

\newcommand{\sll}[1]{\mkern-4mu\mathbin{/\mkern-5mu/}_{\mkern-4mu{#1}}}

\newcommand{\g}{\mathfrak{g}}

\newcommand{\rank}{\mathrm{rank}}

\renewcommand{\exp}{\mathrm{exp}}

\numberwithin{equation}{section}

\title[Abelianization and the Duistermaat--Heckman theorem]{Abelianization and the Duistermaat--Heckman theorem}

\author[Peter Crooks]{Peter Crooks}
\author[Jonathan Weitsman]{Jonathan Weitsman}
\address[Peter Crooks]{Department of Mathematics and Statistics\\ Utah State University \\ 3900 Old Main Hill \\ Logan, UT 84322, USA}
\email{peter.crooks@usu.edu}
\address[Jonathan Weitsman]{Department of Mathematics \\ Northeastern University \\ 360 Huntington Avenue \\ Boston, MA 02115, USA}
\email{j.weitsman@northeastern.edu}

\begin{document}

\begin{abstract}
Fix a compact connected Lie group $G$ with Lie algebra $\g$, as well as a strong Gelfand--Cetlin datum on $\g^*$. Let us also fix a connected symplectic manifold $M$ endowed with an effective Hamiltonian $G$-action and proper moment map. We associate to such information a measure on $\mathbb{R}^{\mathrm{b}}$, where $\mathrm{b}=\frac{1}{2}(\dim G+\rank\hspace{2pt}G)$. We also express the Radon--Nikodym derivative of this measure in terms of the volumes of the symplectic quotients of $M$ by $G$, and thereby prove a non-abelian version of the Duistermaat--Heckman theorem.        
\end{abstract}

\maketitle
\begin{scriptsize}
\tableofcontents
\end{scriptsize}

\section{Introduction}
\subsection{The Duistermaat--Heckman theorem}\label{Subsection: First}
Let $(M,\omega)$ be a $2n$-dimensional connected symplectic manifold. The volume form $\frac{1}{n!}\omega^n$ then determines a measure $\lambda_M$ on $M$. Let us also suppose that $M$ comes equipped with an effective Hamiltonian action of a compact torus $T$, and that this action admits a proper moment map $\mu:M\longrightarrow\mathfrak{t}^*$. One may endow $T$ with the normalized Haar measure for which $T$ has unit volume; this gives rise to a Lebesgue measure on $\mathfrak{t}$, which in turn induces a dual Lebesgue measure $\lambda_{\mathfrak{t}^*}$ on $\mathfrak{t}^*$. A second measure on $\mathfrak{t}^*$ may be obtained by pushing $\lambda_M$ forward along $\mu$; this new measure $\mathrm{DH}_{\mathfrak{t}^*}\coloneqq\mu_{*}\lambda_M$ is defined by
$$\mathrm{DH}_{\mathfrak{t}^*}(B)=\int_{\mu^{-1}(B)}\frac{1}{n!}\omega^n$$ for all Borel subsets $B\subset\mathfrak{t}^*$. 

Duistermaat and Heckman \cite{DuistermaatHeckman} relate $\mathrm{DH}_{\mathfrak{t}^*}$ to $\lambda_{\mathfrak{t}^*}$ by means of the symplectic quotients
$$M\sll{\xi}T\coloneqq\mu^{-1}(\xi)/T,$$
where $\xi\in\mathfrak{t}^*$ ranges over the regular values of $\mu$. Each such quotient is a compact symplectic orbifold, and as such has a well-defined symplectic volume $\mathrm{vol}(M\sll{\xi}T)$. One version of the Duistermaat--Heckman theorem is that $$\mathrm{DH}_{\mathfrak{t}^*}=\rho_{\mathfrak{t}^*}\lambda_{\mathfrak{t}^*},$$ where $\rho_{\mathfrak{t}^*}:\mathfrak{t}^*\longrightarrow[0,\infty)$ is a Radon--Nikodym derivative of $\mathrm{DH}_{\mathfrak{t}^*}$ with respect to $\lambda_{\mathfrak{t}^*}$ satisfying $$\rho_{\mathfrak{t}^*}(\xi)=\mathrm{vol}(M\sll{\xi}T)$$ for all regular values $\xi\in\mathfrak{t}^*$ of $\mu$. The underlying argument is largely based on the existence of appropriate local normal forms for $\mu$, as presented in \cite[Equations (1.7) and (2.1)]{DuistermaatHeckman}. 

\subsection{Gelfand--Cetlin data and abelianization}
Let $G$ be a compact connected Lie group with Lie algebra $\g$. Our paper \cite{CrooksWeitsman} defines what it means for a continuous map $\nu_{\text{big}}:\g^*\longrightarrow\mathbb{R}_{\text{big}}$ and open dense subset $\g^*_{\text{s-reg}}\subset\g^*$ to form a \textit{Gelfand--Cetlin datum} on $\g^*$, where $$\mathrm{b}\coloneqq\frac{1}{2}(\dim G+\rank\hspace{2pt}G)\quad\text{and}\quad\mathbb{R}_{\text{big}}\coloneqq\mathbb{R}^{\mathrm{b}}.$$ Among other things, $\nu_{\text{big}}\big\vert_{\g^*_{\text{s-reg}}}$ is a moment map for a Hamiltonian action of $$\mathbb{T}_{\text{big}}\coloneqq\operatorname{U}(1)^{\mathrm{b}}$$ on the open Poisson submanifold $\g^*_{\text{s-reg}}\subset\g^*$; further details are provided in \cite[Definition 1]{CrooksWeitsman}. The works of Guillemin--Sternberg \cite{GuilleminSternbergGC,GuilleminSternbergThimm} and Hoffman--Lane \cite{Lane} imply that Gelfand--Cetlin on $\g^*$ exist.

Now suppose that $M$ is a Hamiltonian $G$-space with moment map $\mu:M\longrightarrow\g^*$, and that $(\nu_{\text{big}},\g^*_{\text{s-reg}})$ is a Gelfand--Cetlin datum on $\g^*$. The torus $\mathbb{T}_{\text{big}}$ then acts in a Hamiltonian fashion on the open subset $M_{\text{s-reg}}\coloneqq\mu^{-1}(\g^*_{\text{s-reg}})\subset M$ with moment map  $(\nu_{\text{big}}\circ\mu)\big\vert_{M_{\text{s-reg}}}$. The main result of \cite{CrooksWeitsman} is the following \textit{abelianization theorem} for generic symplectic quotients of $M$ by $G$: there is a canonical isomorphism $$M\sll{\xi}G\cong M_{\text{s-reg}}\sll{\nu_{\text{big}}(\xi)}\mathbb{T}_{\text{big}}$$ of stratified symplectic spaces \cite{SjamaarLerman} for each $\xi\in\g^*_{\text{s-reg}}$. 

\subsection{A non-abelian Duistermaat--Heckman theorem}
The purpose of this paper is to show that Gelfand--Cetlin data and the abelianization theorem from \cite{CrooksWeitsman} give rise to a generalization of the Duistermaat--Heckman theorem. We take $M$ to be a connected symplectic manifold equipped with an effective Hamiltonian $G$-action and proper moment map $\mu:M\longrightarrow\g^*$. We also fix a \textit{strong} Gelfand--Cetlin datum (Definition \ref{Definition: Second}) on $\g^*$. The continuous map $$\mu_{\text{big}}\coloneqq\nu_{\text{big}}\circ\mu:M\longrightarrow\mathbb{R}_{\text{big}}$$ is proper and gives rise to a measure $$\mathrm{DH}_{\text{big}}\coloneqq(\mu_{\text{big}})_*\lambda_M$$ on $\mathbb{R}_{\text{big}}$. On the other hand, normalizing the Haar measure on $\mathbb{T}_{\text{big}}$ to give a unit volume induces a Lebesgue measure $\lambda_{\text{big}}$ on $\mathbb{R}_{\text{big}}=\mathrm{Lie}(\mathbb{T}_{\text{big}})^*$. 

To compare $\mathrm{DH}_{\text{big}}$ and $\lambda_{\text{big}}$, we let $\g^*_{\mu,\text{s-reg}}$ denote locus of the regular elements of $\mu$ lying in $\g^*_{\text{s-reg}}$. The subset
$$\mathrm{U}_{\text{big}}\coloneqq\nu_{\text{big}}(\g^*_{\mu,\text{s-reg}})\subset\mathbb{R}_{\text{big}}$$ is open, and allows us to state the following comparison between $\mathrm{DH}_{\text{big}}$ and $\lambda_{\text{big}}$; this is the main result of our paper.

\begin{theorem*}
There is a Radon--Nikodym derivative $\rho_{\emph{big}}:\mathrm{U}_{\emph{big}}\longrightarrow[0,\infty)$ of $\mathrm{DH}_{\emph{big}}$ on $\mathrm{U}_{\emph{big}}$ with respect to $\lambda_{\emph{big}}$, and it satisfies $$\rho_{\emph{big}}(\lambda_{\emph{big}}(\xi))=\mathrm{vol}(M\sll{\xi}G)$$ for all regular values $\xi\in\g^*_{\emph{s-reg}}$ of $\mu$.
\end{theorem*}

If $G$ is a torus, then one can choose a strong Gelfand--Cetlin datum in such a way that $\g^*_{\text{s-reg}}=\g^*$, $\nu_{\text{big}}:\g^*\longrightarrow\mathbb{R}_{\text{big}}$ is a vector space isomorphism, and $\mathrm{U}_{\text{big}}$ is the locus of regular values of $\mu$ in $\mathbb{R}_{\text{big}}=\g^*$. This makes it clear that $\rho_{\text{big}}$ coincides with Duistermaat and Heckman's Radon--Nikodym derivative on regular values of $\mu$. In other words, our result may be interpreted as a generalization of the Duistermaat--Heckman theorem into the realm of non-abelian group actions.

\subsection{Recovering a description of a measure on $\mathfrak{t}_{+}^*$}
Let $T\subset G$ be a maximal torus with Lie algebra $\mathfrak{t}\subset\g$. By choosing positive roots and a $G$-invariant inner product on $\g$, one obtains an inclusion $\mathfrak{t}^*\subset\g^*$ and fundamental Weyl chamber $\mathfrak{t}_{+}^*\subset\mathfrak{t}^*$. Restrict the measure $\lambda_{\mathfrak{t}^*}$ from Section \ref{Subsection: First} to a measure $\lambda_{\mathfrak{t}_{+}^*}$ on $\mathfrak{t}_{+}^*$. Let us also consider the adjoint quotient map $\pi:\g^*\longrightarrow\mathfrak{t}_{+}^*$.

Now let $M$ be a connected symplectic manifold on which $G$ acts effectively and in a Hamiltonian fashion with proper moment map $\mu:M\longrightarrow\g^*$. The measure $$\mathrm{DH}_{\mathfrak{t}_{+}^*}\coloneqq(\pi\circ\mu)_*\lambda_M$$ on $\mathfrak{t}_{+}^*$ is another non-abelian variant of the Duistermaat--Heckman measure, and has received some attention in the literature \cite{GuilleminPrato,CDK,ZJW,GLS}. 

We use the main result of this paper to prove that $\mathrm{DH}_{\mathfrak{t}_{+}^*}$ has a Radon--Nikodym derivative of $$\rho_{\mathfrak{t}_{+}^*}(\xi)=\mathrm{vol}(\mathcal{O}_{\xi})\cdot\mathrm{vol}(M\sll{\xi}G)$$ with respect to $\lambda_{\mathfrak{t}_{+}^*}$ on a suitable open subset of $\mathfrak{t}_{+}^*$, where $\mathcal{O}_{\xi}\subset\g^*$ denotes the coadjoint orbit of $G$ through $\xi\in\mathfrak{t}_{+}^*$ and $\mathrm{vol}(\mathcal{O}_{\xi})$ denotes its symplectic volume. This fact is presumably well-known and/or readily deducible from existing results. 

\subsection{Organization}
The Lie-theoretic foundations of our paper are outlined in Section \ref{Subsection: Lie theory}. This creates the appropriate context for Section \ref{Subsection: Gelfand--Cetlin data and abelianization}, in which we recall the pertinent parts of \cite{CrooksWeitsman}. Strong Gelfand--Cetlin data are then defined and briefly discussed in Section \ref{Subsection: Strong Gelfand--Cetlin data}. In Section \ref{Subsection: Non-abelian Duistermaat--Heckman measures}, we give some context for the non-abelian Duistermaat--Heckman measures $\mathrm{DH}_{\text{big}}$ and $\mathrm{DH}_{\mathfrak{t}_{+}^*}$ mentioned above. The Radon--Nikodym derivatives of $\mathrm{DH}_{\text{big}}$ and $\mathrm{DH}_{\mathfrak{t}_{+}^*}$ are then derived in Section \ref{Subsection: Radon--Nikodym derivatives}.  

\subsection*{Acknowledgements} The authors are grateful to Megumi Harada for encouraging us to pursue this work. P.C. is supported by a Utah State University startup grant, while J.W. is supported by Simons Collaboration Grant \# 579801.

\section{Main results}
\subsection{Lie theory}\label{Subsection: Lie theory}
Let  $G$ be a compact connected Lie group of rank $\ell$ with Lie algebra $\g$ and exponential map $\exp:\g\longrightarrow G$. Write $G_{\xi}\subset G$ and $\g_{\xi}\subset\g$ for the centralizers of $\xi\in\g^*$ with respect to the coadjoint representations of $G$ and $\g$, respectively, and $\mathcal{O}_{\xi}\subset\g^*$ for the coadjoint orbit of $G$ through $\xi$. The latter centralizer is then the Lie algebra of the former. On the other hand, let us consider the locus of regular elements
$$\g^*_{\text{reg}}\coloneqq\{\xi\in\g^*:\dim\g_{\xi}=\ell\}.$$ One knows that $\xi\in\g^*$ is regular if and only if $G_{\xi}$ is a maximal torus of $G$. In this case, $$\Lambda_{\xi}\coloneqq\frac{1}{2\pi}\mathrm{ker}\left(\exp\big\vert_{\g_{\xi}}\right)\subset\g_{\xi}$$ is a free $\mathbb{Z}$-submodule of rank $\ell$.

Fix a maximal torus $T\subset G$ with Lie algebra $\mathfrak{t}\subset\g$, and choose a set $\Phi_{+}\subset\mathfrak{t}^*$ of positive roots. One thereby obtains the fundamental Weyl chamber
$$\mathfrak{t}_{+}\coloneqq\{x\in\mathfrak{t}:\alpha(x)\geq 0\text{ for all }\alpha\in\Phi_{+}\}\subset\mathfrak{t}.$$ Let us also choose a $G$-invariant inner product on $\g$, and use it to identify the adjoint and coadjoint representations of $G$. This identification allows one to identify $\mathfrak{t}_{+}\subset\g$ with a fundamental domain $\mathfrak{t}_{+}^*\subset\g^*$ for the coadjoint action. The \textit{sweeping map} $\pi:\g^*\longrightarrow\mathfrak{t}_{+}^*$ is then defined by the property that $\{\pi(\xi)\}=\mathcal{O}_{\xi}\cap\mathfrak{t}_{+}^*$ for all $\xi\in\g^*$; its fibers are precisely the coadjoint orbits of $G$.

\subsection{Gelfand--Cetlin data and abelianization}\label{Subsection: Gelfand--Cetlin data and abelianization}
Let us set
$$\mathrm{u}\coloneqq\frac{1}{2}(\dim\g-\ell)\quad\text{and}\quad\mathrm{b}\coloneqq\frac{1}{2}(\dim\g+\ell),$$ and define the three tori $$\mathbb{T}_{\text{small}}\coloneqq\operatorname{U}(1)^{\ell},\quad\mathbb{T}_{\text{int}}\coloneqq\operatorname{U}(1)^{\mathrm{u}},\quad\text{and}\quad \mathbb{T}_{\text{big}}\coloneqq\mathbb{T}_{\text{small}}\times\mathbb{T}_{\text{int}}\cong\operatorname{U}(1)^{\mathrm{b}}.$$ By identifying $\mathbb{R}$ with $i\mathbb{R}$ in the usual way, we may regard 
$$\mathbb{R}_{\text{small}}\coloneqq\mathbb{R}^{\ell},\quad\mathbb{R}_{\text{int}}\coloneqq\mathbb{R}^{\mathrm{u}},\quad\text{and}\quad\mathbb{R}_{\text{big}}\coloneqq\mathbb{R}_{\text{small}}\times\mathbb{R}_{\text{int}}\cong\mathbb{R}^{\mathrm{b}}$$ as the duals of the Lie algebras of $\mathbb{T}_{\text{small}}$, $\mathbb{T}_{\text{int}}$, and $\mathbb{T}_{\text{big}}$, respectively.

The following is a rephrased version of \cite[Definition 1]{CrooksWeitsman}.

\begin{definition}\label{Definition: Main definition}
A \textit{Gelfand--Cetlin datum} $(\nu_{\text{big}},\g^*_{\text{s-reg}})$ consists of a continuous map $\nu_{\text{big}}=(\nu_1,\ldots,\nu_{\mathrm{b}}):\g^*\longrightarrow\mathbb{R}_{\text{big}}$ and open dense subset $\g^*_{\text{s-reg}}\subset\g_{\text{reg}}^*$ that satisfy the following conditions:
\begin{itemize}
\item[\textup{(i)}] $\nu_1,\ldots,\nu_{\ell}$ are $G$-invariant on $\g^*$ and smooth on $\g^*_{\text{reg}}$;
\item[\textup{(ii)}] the set of differentials $\{\mathrm{d}_{\xi}\nu_1,\ldots,\mathrm{d}_{\xi}\nu_{\ell}\}$ is a $\mathbb{Z}$-basis of $\Lambda_{\xi}\subset\g_{\xi}$ for all $\xi\in\g^*_{\text{reg}}$;
\item[\textup{(iii)}] $\nu_{\text{big}}\big\vert_{\g^*_{\text{s-reg}}}:\g^*_{\text{s-reg}}\longrightarrow\mathbb{R}_{\text{big}}$ is a smooth submersion and moment map for a Poisson Hamiltonian $\mathbb{T}_{\text{big}}$-space structure on $\g^*_{\text{s-reg}}$;
\item[\textup{(iv)}] $\nu_{\text{big}}\big\vert_{\g^*_{\text{s-reg}}}:\g^*_{\text{s-reg}}\longrightarrow\nu_{\text{big}}(\g^*_{\text{s-reg}})$ is a principal $\mathbb{T}_{\text{int}}$-bundle;
\item[\textup{(v)}] if $M$ is a Hamiltonian $G$-space with moment map $\mu:M\longrightarrow\g^*$, then $$(\nu_{\text{big}}\circ\mu)\big\vert_{\mu^{-1}(\g^*_{\text{s-reg}})}: \mu^{-1}(\g^*_{\text{s-reg}})\longrightarrow\mathbb{R}_{\text{big}}$$ is a moment map for a Hamiltonian $\mathbb{T}_{\text{big}}$-space structure on $\mu^{-1}(\g_{\text{s-reg}})$.
\end{itemize} 
In this case, we will write $$\nu_{\text{small}}\coloneqq (\lambda_1,\ldots,\lambda_{\ell}):\g^*\longrightarrow\mathbb{R}_{\text{small}},\quad\nu_{\text{int}}\coloneqq(\lambda_{\ell+1},\ldots,\lambda_{\mathrm{b}}):\g^*\longrightarrow\mathbb{R}_{\text{int}},$$ $$M_{\text{s-reg}}\coloneqq\mu^{-1}(\g^*_{\text{s-reg}}),\quad\text{and}\quad\mu_{\text{big}}\coloneqq\nu_{\text{big}}\circ\mu:M\longrightarrow\mathbb{R}_{\text{big}}.$$
\end{definition}

\begin{remark}\label{Remark: Torus GC}
Suppose that $G$ is a torus, i.e. $G=T$ and $\g=\mathfrak{t}$. Let $\{\nu_1,\ldots,\nu_{\ell}\}$ be a $\mathbb{Z}$-basis of $\frac{1}{2\pi}\mathrm{ker}\left(\exp:\mathfrak{g}\longrightarrow G\right)\subset\mathfrak{g}$. It is straightforward to check that $\nu_{\text{big}}\coloneqq(\nu_1,\ldots,\nu_{\ell}):\g^*\longrightarrow\mathbb{R}^{\ell}$ and $\g^*_{\text{s-reg}}\coloneqq\g^*$ form a Gelfand--Cetlin datum.
\end{remark}

\begin{remark}
The situation is considerably more subtle if one requires $G$ to be simple. Examples of Gelfand--Cetlin data then include Guillemin and Sternberg's Gelfand--Cetlin systems on the special unitary and special orthogonal Lie algebras \cite{GuilleminSternbergGC,GuilleminSternbergThimm}, as well as Hoffman and Lane's recent generalizations of these systems to arbitrary Lie type \cite{Lane}.
\end{remark}

The main result in \cite{CrooksWeitsman} is the following comparison between the symplectic quotients of a Hamiltonian $G$-space $M$ and those of the Hamiltonian $\mathbb{T}_{\text{big}}$-space $M_{\text{s-reg}}$.

\begin{theorem}\label{Theorem: Abelianization}
Let $(\nu_{\emph{big}},\g^*_{\emph{s-reg}})$ be a Gelfand--Cetlin datum and $M$ a Hamiltonian $G$-space. Given any $\xi\in\g^*_{\emph{s-reg}}$, there is a canonical isomorphism $$M\sll{\xi}G\cong M_{\emph{s-reg}}\sll{\nu_{\emph{big}}(\xi)}\mathbb{T}_{\emph{big}}$$ of stratified symplectic spaces. 
\end{theorem}

The definition of a \textit{stratified symplectic space} is due to Sjamaar--Lerman \cite{SjamaarLerman}. This strict generalization of a symplectic manifold liberates us from making assumptions to ensure that the symplectic quotients $M\sll{\xi}G$ and $M_{\text{s-reg}}\sll{\nu_{\text{big}}(\xi)}\mathbb{T}_{\text{big}}$ are genuine symplectic manifolds.

\subsection{Strong Gelfand--Cetlin data}\label{Subsection: Strong Gelfand--Cetlin data}
In order to prove the main results of this manuscript, we will need the following strengthening of Definition \ref{Definition: Main definition}.

\begin{definition}\label{Definition: Second} A Gelfand--Cetlin datum $(\nu_{\text{big}},\g^*_{\text{s-reg}})$ will be called \textit{strong} if it satisfies the following additional conditions:
\begin{itemize}
\item[(vi)] if $\xi,\eta\in\g^*$ and $\nu_{\text{small}}(\xi)=\nu_{\text{small}}(\eta)$, then $\mathcal{O}_{\xi}=\mathcal{O}_{\eta}$;
\item[(vii)] $\nu_{\text{big}}$ is proper;
\item[(viii)] $\g^*_{\text{s-reg}}$ is a union of fibers of $\nu_{\text{big}}$;
\item[(ix)] if $\xi\in\g^*_{\text{s-reg}}$, then $\mathcal{O}_{\xi}\cap\g^*_{\text{s-reg}}$ is dense in $\mathcal{O}_{\xi}$.
\end{itemize}
\end{definition}

\begin{remark}\label{Remark: Strong GC}
Examples of strong Gelfand--Cetlin data include the data described in Remark \ref{Remark: Torus GC}, as well as Guillemin and Sternberg's Gelfand--Cetlin systems on the special unitary and special orthogonal Lie algebras \cite{GuilleminSternbergGC,GuilleminSternbergThimm}. 
\end{remark}

\subsection{Non-abelian Duistermaat--Heckman measures}\label{Subsection: Non-abelian Duistermaat--Heckman measures}
Let $M$ be a connected symplectic manifold carrying an effective Hamiltonian action of $G$ and proper moment map $\mu:M\longrightarrow\g^*$. Define the $G$-invariant moment map to be the composition $$\overline{\mu}\coloneqq\pi\circ\mu:M\longrightarrow\mathfrak{t}_{+}^*.$$ Let us also suppose that $(\nu_{\text{big}},\g^*_{\text{s-reg}})$ is a strong Gelfand--Cetlin datum. The following five measures naturally arise in this context:
\begin{itemize}
\item the measure $\lambda_{\mathfrak{t}_{+}^*}$ on $\mathfrak{t}_{+}^*$ induced by normalizing the Haar measure on $T$ to give a unit volume, forming the resulting Lebesgue measure on $\mathfrak{t}$, taking the induced Lebesgue measure on $\mathfrak{t}^*$, and restricting to $\mathfrak{t}_{+}^*$;
\item the measure $\lambda_{\text{big}}$ on $\mathbb{R}_{\text{big}}$ obtained by normalizing the Haar measure on $\operatorname{U}(1)^{\mathrm{b}}$, and proceeding analogously to the above to get a measure on $\mathbb{R}_{\text{big}}=\mathrm{Lie}(\operatorname{U}(1)^{\mathrm{b}})^*$;
\item the Liouville measure $\lambda_M$ on $M$ associated to the volume form $\frac{1}{n!}\omega^n$, where $\omega$ is the symplectic form on $M$ and $2n=\dim M$;
\item the measure on $\mathfrak{t}_{+}^*$ obtained by pushing $\lambda_M$ forward along $\overline{\mu}:M\longrightarrow\mathfrak{t}_{+}^*$, i.e $\mathrm{DH}_{\mathfrak{t}_{+}^*}\coloneqq\overline{\mu}_{*}\lambda_M$.
\item the measure on $\mathbb{R}_{\text{big}}$ obtained by pushing $\lambda_M$ forward along $\mu_{\text{big}}:M\longrightarrow\mathbb{R}_{\text{big}}$, i.e $\mathrm{DH}_{\text{big}}\coloneqq(\mu_{\text{big}})_{*}\lambda_M$;
\end{itemize}

\begin{remark}\label{Remark: Torus measures}
Suppose that $G=T$ is a torus, and let all objects and notation be as discussed in Remark \ref{Remark: Torus GC}. Observe that $\mathfrak{t}_{+}^*=\g^*$, $\mathbb{R}_{\text{big}}=\mathbb{R}^{\ell}$, and that $\nu_{\text{big}}:\mathfrak{g}^*\longrightarrow\mathbb{R}_{\text{big}}$ is a vector space isomorphism. By means of the previous sentence, we may regard both $\mathrm{DH}_{\mathfrak{t}^*_{+}}$ and $\mathrm{DH}_{\text{big}}$ as measures on $\mathbb{R}^{\ell}$. These two measures are readily seen to coincide.  This coincident measure is the subject of Duistermaat and Heckman's work \cite{DuistermaatHeckman}, and so often called the \textit{Duistermaat--Heckman measure}.
\end{remark}

The topological spaces $\mathfrak{t}^*_{+}$ and $\mathbb{R}_{\text{big}}$ are homeomorphic if and only if $G$ is a torus. In particular, the measures discussed in Remark \ref{Remark: Torus measures} are defined on different spaces in the case of our arbitrary compact connected Lie group $G$. One may therefore regard $\mathrm{DH}_{\mathfrak{t}^*_{+}}$ and $\mathrm{DH}_{\text{big}}$ as two different generalizations of the classical Duistermaat--Heckman measure, each into the realm of Hamiltonian actions by non-abelian groups. This motivates the following definition.

\begin{definition}
Let $M$ be a connected Hamiltonian $G$-space with an effective $G$-action and proper moment map $\mu:M\longrightarrow\g^*$.
\begin{itemize}
\item[(i)] We call $\mathrm{DH}_{\mathfrak{t}_{+}^*}\coloneqq\overline{\mu}_{*}\lambda_M$ the \textit{non-abelian Duistermaat--Heckman measure} on $\mathfrak{t}_{+}^*$ associated to $M$.
\item[(ii)] Suppose that $(\nu_{\text{big}},\g^*_{\text{s-reg}})$ is a strong Gelfand--Cetlin datum. We call $\mathrm{DH}_{\text{big}}\coloneqq(\mu_{\text{big}})_{*}\lambda_M$ the \textit{non-abelian Duistermaat--Heckman measure} on $\mathbb{R}_{\text{big}}$ associated to $M$ and $(\nu_{\text{big}},\g^*_{\text{s-reg}})$.
\end{itemize}
\end{definition}

The measure $\mathrm{DH}_{\mathfrak{t}_{+}^*}$ and variants thereof have been studied quite extensively \cite{GuilleminPrato,CDK,ZJW,Duistermaat,GLS,PratoWu,Witten,JeffreyKirwan}.

\subsection{Radon--Nikodym derivatives}\label{Subsection: Radon--Nikodym derivatives}
Let $(\nu_{\text{big}},\g^*_{\text{s-reg}})$ be a strong Gelfand--Cetlin datum, and $M$ a connected Hamiltonian $G$-space with proper moment map $\mu:M\longrightarrow\g^*$ and effective $G$-action. Consider the open subset $$\g^*_{\mu,\text{s-reg}}\coloneqq\{\xi\in\g^*_{\text{s-reg}}:\mathrm{d}_{m}\mu\text{ is surjective for all }m\in \mu^{-1}(\xi)\}$$ of regular values of $\mu$ lying in $\g^*_{\text{s-reg}}$. 

\begin{lemma}\label{Lemma: Useful}
The subset $\g^*_{\mu,\emph{s-reg}}$ is a union of fibers of $\nu_{\emph{big}}$. Each such fiber is contained in a regular coadjoint orbit of $G$.
\end{lemma}

\begin{proof}
Let $\xi\in\g^*$ and $\eta\in\g^*_{\mu,\text{s-reg}}$ be such that $\nu_{\text{big}}(\xi)=\nu_{\text{big}}(\eta)$. Definition \ref{Definition: Second}(viii) implies that $\xi\in\g^*_{\text{s-reg}}$, while Definition \ref{Definition: Second}(vi) tells us that $\mathcal{O}_{\xi}=\mathcal{O}_{\eta}$. Since $\mu$ is $G$-equivariant and $\eta$ is a regular value of $\mu$, the condition $\mathcal{O}_{\xi}=\mathcal{O}_{\eta}$ forces $\xi$ to be a regular value of $\mu$. We conclude that $\xi\in\g^*_{\mu,\text{s-reg}}$. Our arguments also imply that $\nu_{\text{big}}^{-1}(\nu_{\text{big}}(\eta))\subset\mathcal{O}_{\eta}$. These last two sentences suffice to complete the proof.
\end{proof}

Now consider the image
$$\mathrm{U}_{\text{big}}\coloneqq\nu_{\text{big}}(\g^*_{\mu,\text{s-reg}})$$ of the open subset $\g^*_{\mu,\text{s-reg}}\subset\g^*_{\text{s-reg}}$ under $\nu_{\text{big}}$; it is open in $\mathbb{R}_{\text{big}}$ by Definition \ref{Definition: Main definition}(iii). Given any $x\in\mathrm{U}_{\text{big}}$ and $\xi,\eta\in\nu_{\text{big}}^{-1}(x)$, Lemma \ref{Lemma: Useful} and the $G$-equivariance of $\mu$ imply that $M\sll{\xi}G\cong M\sll{\eta}G$ as symplectic orbifolds. It follows that the symplectic volume $$\rho_{\text{big}}(x)\coloneqq\mathrm{vol}(M\sll{\xi}G)$$ does not depend on the choice of $\xi\in\nu_{\text{big}}^{-1}(x)$. We thereby obtain a well-defined function $\rho_{\text{big}}:U_{\text{big}}\longrightarrow[0,\infty)$; it will turn out to be a Radon--Nikodym derivative of the non-abelian Duistermaat--Heckman measure $\mathrm{DH}_{\text{big}}$ with respect to $\lambda_{\text{big}}$, in the following sense. 

\begin{theorem}\label{Theorem: Main theorem}
The function $\rho_{\emph{big}}$ is a Radon--Nikodym derivative of $\mathrm{DH}_{\emph{big}}$ on the open subset $\mathrm{U}_{\emph{big}}\subset\mathbb{R}_{\emph{big}}$ with respect to $\lambda_{\emph{big}}$, i.e. $\mathrm{DH}_{\emph{big}}=\rho_{\emph{big}}\lambda_{\emph{big}}$ on $\mathrm{U}_{\emph{big}}$.
\end{theorem}

\begin{proof}
The properness of $\mu:M\longrightarrow\g$ implies that of its pullback $\mu^{-1}(\g^*_{\mu,\text{s-reg}})\longrightarrow\g^*_{\mu,\text{s-reg}}$ to $\g^*_{\mu,\text{s-reg}}\subset\g^*$. It is also clear that $\mu^{-1}(\g^*_{\mu,\text{s-reg}})\longrightarrow\g^*_{\mu,\text{s-reg}}$ is a submersion. On the other hand, Definition \ref{Definition: Main definition}(iv) and Lemma \ref{Lemma: Useful} tell us that $\nu_{\text{big}}$ restricts to a proper submersion $\g^*_{\mu,\text{s-reg}}\longrightarrow \mathrm{U}_{\text{big}}$. The last few sentences force the composite map 
$$\mu^{-1}(\g^*_{\mu,\text{s-reg}})\longrightarrow\g^*_{\mu,\text{s-reg}}\longrightarrow\mathrm{U}_{\text{big}}$$ to be a proper submersion $\mu^{-1}(\g^*_{\mu,\text{s-reg}})\longrightarrow\mathrm{U}_{\text{big}}$. This proper submersion is precisely the pullback of $\mu_{\text{big}}:M\longrightarrow\mathbb{R}_{\text{big}}$ to $\mathrm{U}_{\text{big}}\subset\mathbb{R}_{\text{big}}$. It is also the pullback of the $\mathbb{T}_{\text{big}}$-moment map $\mu_{\text{big}}\big\vert_{M_{\text{s-reg}}}:M_{\text{s-reg}}\longrightarrow\mathbb{R}_{\text{big}}$ to $\mathrm{U}_{\text{big}}$. The local normal forms of moment maps used in the proof of the Duistermaat--Heckman theorem \cite[Section 3]{DuistermaatHeckman} now yield
$$\mathrm{DH}_{\text{big}}=(\mu_{\text{big}})_{*}\lambda_{M}=\phi\lambda_{\text{big}}$$ on $\mathrm{U}_{\text{big}}$, where $\phi:\mathrm{U}_{\text{big}}\longrightarrow [0,\infty)$ is defined by $\phi(\eta)=\mathrm{vol}(M_{\text{s-reg}}\sll{\eta}\mathbb{T}_{\text{big}})$. Theorem \ref{Theorem: Abelianization} and the definition of $\rho_{\text{big}}$ immediately imply that $\phi=\rho_{\text{big}}$, completing the proof.
\end{proof}

Consider the pullback $\pi_{\text{reg}}:\g^*_{\text{reg}}\longrightarrow(\mathfrak{t}_{+}^*)^{\circ}$ of $\pi:\g^*\longrightarrow\mathfrak{t}_{+}^*$ to the interior $(\mathfrak{t}_{+}^*)^{\circ}$ of $\mathfrak{t}_{+}^*$. Since $\pi_{\text{reg}}$ is a submersion, the image $$\mathrm{U}_{\mathfrak{t}^*_{+}}\coloneqq\pi(\g^*_{\mu,\text{s-reg}})$$ is an open subset of $(\mathfrak{t}_{+}^*)^{\circ}$. It is also clear that $\mathrm{U}_{\mathfrak{t}^*_{+}}$ consists of regular values of $\mu$. We may therefore define $\rho_{\mathfrak{t}_{+}^*}:\mathrm{U}_{\mathfrak{t}^*_{+}}\longrightarrow[0,\infty)$ by $$\rho_{\mathfrak{t}_{+}^*}(\xi)\coloneqq\mathrm{vol}(\mathcal{O}_{\xi})\cdot\mathrm{vol}(M\sll{\xi}G),$$ where $\mathrm{vol}(\mathcal{O}_{\xi})$ denotes the symplectic volume of $\mathcal{O}_{\xi}$. 

\begin{corollary}
The function $\rho_{\mathfrak{t}_{+}^*}$ is a Radon--Nikodym derivative of $\mathrm{DH}_{\mathfrak{t}_{+}^*}$ on the open subset $\mathrm{U}_{\mathfrak{t}_{+}^*}\subset\mathfrak{t}_{+}^*$ with respect to $\lambda_{\mathfrak{t}_{+}^*}$, i.e. $\mathrm{DH}_{\mathfrak{t}_{+}^*}=\rho_{\mathfrak{t}_{+}^*}\lambda_{\mathfrak{t}_{+}^*}$ on $\mathrm{U}_{\mathfrak{t}_{+}^*}$.
\end{corollary}

\begin{proof}
Consider the pullbacks $\overline{\mu}^{-1}(\mathrm{U}_{\mathfrak{t}_{+}^*})\longrightarrow\pi^{-1}(\mathrm{U}_{\mathfrak{t}_{+}^*})$ and $\pi^{-1}(\mathrm{U}_{\mathfrak{t}^*_{+}})\longrightarrow \mathrm{U}_{\mathfrak{t}^*_{+}}$ of $\mu:M\longrightarrow\g^*$ and $\pi:\g^*\longrightarrow\mathfrak{t}_{+}^*$ to the open subsets $\pi^{-1}(\mathrm{U}_{\mathfrak{t}_{+}^*})\subset\g^*$ and $\mathrm{U}_{\mathfrak{t}^*_{+}}\subset\mathfrak{t}_{+}^*$, respectively. The latter pullback map is evidently a proper submersion, while the former is clearly proper. To see that the former is also a submersion, recall that $\mathrm{U}_{\mathfrak{t}_{+}^*}$ consists of regular values of $\mu$.  Let us also observe that $\pi^{-1}(\mathrm{U}_{\mathfrak{t}_{+}^*})$ is the union of the adjoint orbits through the elements of $\mathrm{U}_{\mathfrak{t}_{+}^*}$. The $G$-equivariance of $\mu$ then forces each element of $\pi^{-1}(\mathrm{U}_{\mathfrak{t}_{+}^*})$ to be a regular value of $\mu$. In particular, $\overline{\mu}^{-1}(\mathrm{U}_{\mathfrak{t}_{+}^*})\longrightarrow\pi^{-1}(\mathrm{U}_{\mathfrak{t}_{+}^*})$ is also a proper submersion. The composite map $$\overline{\mu}^{-1}(\mathrm{U}_{\mathfrak{t}_{+}^*})\longrightarrow\pi^{-1}(\mathrm{U}_{\mathfrak{t}_{+}^*})\longrightarrow \mathrm{U}_{\mathfrak{t}_{+}^*}$$ is therefore a proper submersion. This amounts to stating that $\overline{\mu}$ restricts to a proper submersion $\overline{\mu}^{-1}(\mathrm{U}_{\mathfrak{t}_{+}^*})\longrightarrow\mathrm{U}_{\mathfrak{t}_{+}^*}$.

Let us define a map $\varphi:\nu_{\text{big}}(\pi^{-1}(\mathrm{U}_{\mathfrak{t}_{+}^*}))\longrightarrow\mathrm{U}_{\mathfrak{t}_{+}^*}$ as follows: $\varphi(x_1,\ldots,x_{\mathrm{b}})$ is the unique element of $\mathrm{U}_{\mathfrak{t}_{+}^*}$ satisfying $f_i(\varphi(x_1,\ldots,x_{\mathrm{b}}))=x_i$ for all $i\in\{1,\ldots,\ell\}$. This gives rise to a commutative diagram \[
\begin{tikzcd}
 && \nu_{\text{big}}(\pi^{-1}(\mathrm{U}_{\mathfrak{t}_{+}^*})) \arrow[dr,"\varphi"] \\
 & \pi^{-1}(\mathrm{U}_{\mathfrak{t}_{+}^*}) \arrow[ur, "\nu_{\text{big}}"] \arrow[rr, "\pi"] && \mathrm{U}_{\mathfrak{t}_{+}^*}\\
&& \overline{\mu}^{-1}(\mathrm{U}_{\mathfrak{t}_{+}^*})\arrow[ul,"\mu"]\arrow[ur,swap,"\overline{\mu}"]
\end{tikzcd}
.\] At the same time, Definition \ref{Definition: Second}(ix) implies that the open subset $\g^*_{\mu,\text{s-reg}}\subset\pi^{-1}(\mathrm{U}_{\mathfrak{t}_{+}^*})$ is dense. It follows that $\mathrm{U}_{\text{big}}=\nu_{\text{big}}(\g^*_{\mu,\text{s-reg}})$ is open and dense in $\nu_{\text{big}}(\pi^{-1}(\mathrm{U}_{\mathfrak{t}_{+}^*}))$. We conclude that $\mathrm{DH}_{\mathfrak{t}_{+}^*}\big\vert_{\mathrm{U}_{\mathfrak{t}_{+}^*}}$ is the push-forward of $\mathrm{DH}_{\text{big}}\big\vert_{\mathrm{U}_{\text{big}}}$ along $\varphi\big\vert_{\mathrm{U}_{\text{big}}}$. Theorem \ref{Theorem: Main theorem} now reduces us to proving that the fiber of $\varphi\big\vert_{\mathrm{U}_{\text{big}}}$ over each $\xi\in\mathrm{U}_{\mathfrak{t}_{+}^*}$ has Euclidean volume equal to the symplectic volume of $\mathcal{O}_{\xi}$. To this end, we have
\begin{align*}\varphi\big\vert_{\mathrm{U}_{\text{big}}}^{-1}(\xi) & = \varphi^{-1}(\xi)\cap\mathrm{U}_{\text{big}}\\ & = \{(f_1(\xi),\ldots,f_{\ell}(\xi))\}\times\nu_{\mathrm{int}}(\mathcal{O}_{\xi}\cap\g^*_{\mu,\text{s-reg}})\subset\mathbb{R}_{\text{small}}\times\mathbb{R}_{\text{int}}=\mathbb{R}_{\text{big}} \\ & = \{(f_1(\xi),\ldots,f_{\ell}(\xi))\}\times\nu_{\mathrm{int}}(\mathcal{O}_{\xi}\cap\g^*_{\text{s-reg}})\subset\mathbb{R}_{\text{small}}\times\mathbb{R}_{\text{int}}=\mathbb{R}_{\text{big}}\end{align*} for all $\xi\in\mathrm{U}_{\mathfrak{t}_{+}^*}$. It now suffices to observe that the Euclidean volume of $\nu_{\mathrm{int}}(\mathcal{O}_{\xi}\cap\g^*_{\text{s-reg}})\subset\mathbb{R}_{\text{int}}$ coincides with the symplectic volume of $\mathcal{O}_{\xi}$, as follows from Definition \ref{Definition: Second}(ix) and \cite[Proposition 3]{CrooksWeitsman}. 
\end{proof}

\bibliographystyle{acm} 
\bibliography{DH}

\end{document}